\documentclass[wide]{llncs}
\usepackage[english]{babel}
\usepackage{amssymb,amsmath,hyperref}
\usepackage{bbold}
\usepackage{amsrefs}
\usepackage{stackrel}
\usepackage{mathrsfs}
\usepackage{cleveref, enumerate}

\usepackage{transfer}

\newtheorem{thm}{Theorem}

\newtheorem{cor}[thm]{Corollary}
\newtheorem{defi}[thm]{Definition}
\newtheorem{rem}[thm]{Remark}
\newtheorem{nota}[thm]{Notation}

\newcommand\be{\begin{equation}}
\newcommand\ee{\end{equation}}

\newbox\gnBoxA
\newdimen\gnCornerHgt
\setbox\gnBoxA=\hbox{$\ulcorner$}
\global\gnCornerHgt=\ht\gnBoxA
\newdimen\gnArgHgt

\def\Godelnum #1{%
	\setbox\gnBoxA=\hbox{$#1$}%
	\gnArgHgt=\ht\gnBoxA%
	\ifnum \gnArgHgt<\gnCornerHgt
		\gnArgHgt=0pt%
	\else
		\advance \gnArgHgt by -\gnCornerHgt%
	\fi
	\raise\gnArgHgt\hbox{$\ulcorner$} \box\gnBoxA %
		\raise\gnArgHgt\hbox{$\urcorner$}}
\def\bdefi{\begin{defi}\rm}
\def\edefi{\end{defi}}
\def\bnota{\begin{nota}\rm}
\def\enota{\end{nota}}
\def\brem{\begin{rem}\rm}
\def\erem{\end{rem}}

\def\RCA{\textup{\textsf{RCA}}}

\def\RCAo{\textup{\textsf{RCA}}_{0}^{\omega}}
\def\RCAO{\textup{\textsf{RCA}}_{0}^{\Lambda}}

\def\WKL{\textup{\textsf{WKL}}}
\def\ORD{\textup{\textsf{ORD}}}

\def\UWKL{\textup{\textsf{UWKL}}}
\def\UORD{\textup{\textsf{UORD}}}
\def\UDRO{\textup{\textsf{UDRO}}}

\def\T{\mathcal{T}}

\def\bye{\end{document}}

\def\P{\textup{\textsf{P}}}

\def\R{{\mathbb  R}}

\def\FAN{\textup{\textsf{FAN}}}
\def\UFAN{\textup{\textsf{UFAN}}}

\def\R{{\mathbb{R}}}
\def\({\textup{(}}
\def\){\textup{)}}

\def\st{\textup{st}}
\def\nst{\textup{nst}}
\def\asa{\leftrightarrow}

\def\di{\rightarrow}

\def\epsa{\varepsilon_{\alpha}}

\def\paai{\Pi_{1}^{0}\textup{-\textsf{TRANS}}}

\def\QFAC{\textup{\textsf{QF-AC}}}

\def\MU{\textup{\textsf{MU}}}

\def\HAC{\textup{\textsf{HAC}}}

\def\INT{\textup{\textsf{int}}}
\def\pw{\textup{\textsf{pw}}}

\numberwithin{equation}{section}
\numberwithin{thm}{section}

\begin{document}
\title{The effective content of Reverse Nonstandard Mathematics}
\subtitle{and the nonstandard content of effective Reverse Mathematics}
\author{Sam Sanders\thanks{This research was supported by the following funding bodies: FWO Flanders, the John Templeton Foundation, the Alexander von Humboldt Foundation, and the Japan Society for the Promotion of Science.}}
\institute{Munich Center for Mathematical Philosophy, LMU Munich, Germany \& Department of Mathematics, Ghent University
\email{sasander@me.com}}

\maketitle

\begin{abstract}
%Reverse Mathematics (RM) is a program in the foundations of mathematics founded by Friedman (\cite{fried, fried2}) and developed extensively by Simpson and others (\cite{simpson2, simpson1}).  Kohlenbach (\cite{kohlenbach2}) has introduced \emph{higher-order} RM, a framework which enables the study of \emph{uniform} versions of theorems, i.e.\ where a functional computes the object claimed to exist by the original theorem.  
The aim of this paper is to highlight a hitherto unknown \emph{computational} aspect of Nonstandard Analysis pertaining to Reverse Mathematics (RM).  
In particular, we shall establish RM-equivalences \emph{between theorems from Nonstandard Analysis} in a fragment of Nelson's \emph{internal set theory}.   
We then extract \emph{primitive recursive} terms from G\"odel's system $\textsf{T}$ (not involving Nonstandard Analysis) from the proofs of the aforementioned nonstandard equivalences.  The resulting terms turn out to be witnesses for \emph{effective}$^{\ref{flruk}}$ equivalences in Kohlenbach's \emph{higher-order} RM.  In other words, from an RM-equivalence in Nonstandard Analysis, we can extract the associated effective higher-order RM-equivalence \emph{which does not involve Nonstandard Analysis anymore}.  
Finally, we show that certain effective equivalences in turn give rise to the original nonstandard theorems from which they were derived.      
\end{abstract}

\section{Introduction}\label{intro}
\subsection{Aim of this paper}
In two words, this paper deals with a new \emph{computational} aspect of Nonstandard Analysis pertaining to {Reverse Mathematics} (RM), in line with the results in \cite{samzoo, samGH, sambon, fegas}.  
In particular, we shall prove certain RM-equivalences in (a fragment of) Nelson's \emph{internal set theory} (\cite{wownelly}), and use the framework from \cite{brie} as sketched in Section \ref{PIPI}, to obtain \emph{effective}\footnote{An implication $(\exists \Phi)A(\Phi)\di (\exists \Psi)B(\Psi)$ is \emph{effective} if there is a term $t$ in the language such that additionally $(\forall \Phi)[A(\Phi)\di B(t(\Phi))]$, i.e.\ $\Psi$ can be effectively defined in terms of $\Phi$; See also Definition \ref{effimp}.\label{flruk}} equivalences in Kohlenbach's \emph{higher-order} RM (\cite{kohlenbach2}), where the latter \emph{does not} involve Nonstandard Analysis.  Perhaps surprisingly, we also show that from certain effective equivalences, the original nonstandard equivalences can be re-obtained.  
In other words, there is `a two-way street' between higher-order RM, and the RM of Nonstandard Analysis, as suggested by the title.  
We refer to \cite{simpson1, simpson2} for an overview of RM.      

\subsection{Introducing internal set theory}\label{IST}
In Nelson's \emph{syntactic} approach to Nonstandard Analysis (\cite{wownelly}), a new predicate `st($x$)', read as `$x$ is standard' is added to the language of \textsf{ZFC}, the usual foundations of mathematics.  
The notations $(\forall^{\st}x)$ and $(\exists^{\st}y)$ are short for $(\forall x)(\st(x)\di \dots)$ and $(\exists y)(\st(y)\wedge \dots)$.  A formula is called \emph{internal} if it does not involve `st', and \emph{external} otherwise.   
The three external axioms \emph{Idealisation}, \emph{Standard Part}, and \emph{Transfer} govern the new predicate `st';  They are respectively defined\footnote{The superscript `fin' in \textsf{(I)} means that $x$ is finite, i.e.\ its number of elements are bounded by a natural number.} as:  
\begin{enumerate}
\item[\textsf{(I)}] $(\forall^{\st~\textup{fin}}x)(\exists y)(\forall z\in x)\varphi(z,y)\di (\exists y)(\forall^{\st}x)\varphi(x,y)$, for internal $\varphi$ with any (possibly nonstandard) parameters.  
\item[\textsf{(S)}] $(\forall x)(\exists^{\st}y)(\forall^{\st}z)\big((z\in x\wedge \varphi(z))\asa z\in y\big)$, for internal $\varphi$.
\item[\textsf{(T)}] $(\forall^{\st}x)\varphi(x, t)\di (\forall x)\varphi(x, t)$, where $\varphi$ is internal,  $t$ captures \emph{all} parameters of $\varphi$, and $t$ is standard.  
\end{enumerate}
Nelson's system \textsf{IST} is \textsf{ZFC} extended with the aforementioned external axioms;  
\textsf{IST} is a conservative extension of \textsf{ZFC} for the internal language, as proved in \cite{wownelly}.    % i.e.\ without `st'. 
G\"odel's system $\textsf{T}$ (\cite{avi2}) extended with fragments of the external axioms of \textsf{IST}, is studied in \cite{brie}.  
In particular, the latter studies nonstandard extensions of the internal systems \textsf{E-HA}$^{\omega}$ and $\textsf{E-PA}^{\omega}$, respectively \emph{Heyting and Peano arithmetic in all finite types and the axiom of extensionality}.       
We refer to \cite{brie}*{\S2.1} for the details of the latter (mainstream in mathematical logic) systems.  

\medskip

As to notation, in the aforementioned systems of higher-order arithmetic, each variable $x^{\rho}$ comes equipped with a superscript denoting its type, which is however often implicit.  
As to the coding of multiple variables, the type $\rho^{*}$ is the type of finite sequences of type $\rho$, a notational device used in \cite{brie} and this paper;  Underlined variables $\underline{x}$ consist of multiple variables of (possibly) different type.  
%
%\medskip
%
%In the next section, we introduce the system $\P$ assuming familiarity with the higher-type framework of G\"odel's $\textsf{T}$ (See e.g.\ \cite{brie}*{\S2.1}).    
\subsection{A fragment of internal set theory based on G\"odel's \textsf{T}}\label{PIPI}
The system $\P$ consist of the following axioms, starting with the basic ones.  
\bdefi\label{debs}[Basic axioms of $\P$]
\begin{enumerate}
\item The system \textsf{E-PA}$^{\omega*}$ be the definitional extension of \textsf{E-PA}$^{\omega}$ with types for finite sequences as in \cite{brie}*{\S2}. 
\item The set $\T^{*}$ is the collection of all the constants in the language of $\textsf{E-PA}^{\omega*}$.  
\item The external induction axiom \textsf{IA}$^{\st}$ is  % as follows.  
\be\tag{\textsf{IA}$^{\st}$}
\Phi(0)\wedge(\forall^{\st}n^{0})(\Phi(n) \di\Phi(n+1))\di(\forall^{\st}n^{0})\Phi(n).     
\ee
\item The system $ \textsf{E-PA}^{\omega*}_{\st} $ is defined as $ \textsf{E-PA}^{\omega{*}} + \T^{*}_{\st} + \textsf{IA}^{\st}$, where $\T^{*}_{\st}$
consists of the following axiom schemas.
\begin{enumerate}
\item The schema\footnote{The language of $\textsf{E-PA}_{\st}^{\omega*}$ contains a symbol $\st_{\sigma}$ for each finite type $\sigma$, but the subscript is always omitted.  Hence $\T^{*}_{\st}$ is an \emph{axiom schema} and not an axiom.\label{omit}} $\st(x)\wedge x=y\di\st(y)$,
\item The schema providing for each closed term $t\in \T^{*}$ the axiom $\st(t)$.
\item The schema $\st(f)\wedge \st(x)\di \st(f(x))$.
\end{enumerate}
\end{enumerate}
%In order to define the system $\P$, we have to discuss some of the external axioms studied in \cite{brie}.  
Secondly, Nelson's axiom \emph{Standard part} is weakened in \cite{brie} to $\HAC_{\INT}$:
\be\tag{$\HAC_{\INT}$}
(\forall^{\st}x^{\rho})(\exists^{\st}y^{\tau})\varphi(x, y)\di (\exists^{\st}F^{\rho\di \tau^{*}})(\forall^{\st}x^{\rho})(\exists y^{\tau}\in F(x))\varphi(x,y),
\ee
where $\varphi$ is any internal formula.  Note that $F$ only provides a \emph{finite sequence} of witnesses to $(\exists^{\st}y)$, explaining its name \emph{Herbrandized Axiom of Choice}.      
Thirdly,  Nelson's axiom idealisation \textsf{I} appears in \cite{brie} as follows:  
\be\tag{\textsf{I}}
(\forall^{\st} x^{\sigma^{*}})(\exists y^{\tau} )(\forall z^{\sigma}\in x)\varphi(z,y)\di (\exists y^{\tau})(\forall^{\st} x^{\sigma})\varphi(x,y), 
%?stx?? ?y? ?x??? x???x?,y????y? ?stx? ?(x,y).
\ee
where $\varphi$ is any internal formula.
\edefi
%Hence, the system $ \textsf{E-PA}^{\omega*}_{\st} $ is a trivial extension of \textsf{E-PA}$^{\omega^{*}}$.  
For the full system $\P\equiv \textsf{E-PA}^{\omega*}_{\st} +\HAC_{\INT} +\textsf{I}$, we have the following theorem.  
Here, the superscript `$S_{\st}$' is the syntactic translation defined in \cite{brie}*{Def.\ 7.1}, and also listed starting with \eqref{dombu} in the proof of Corollary \ref{consresultcor}.    
\begin{thm}\label{consresult}
Let $\Phi(\tup a)$ be a formula in the language of \textup{\textsf{E-PA}}$^{\omega*}_{\st}$ and suppose $\Phi(\tup a)^\Sh\equiv\forallst \tup x \, \existsst \tup y \, \varphi(\tup x, \tup y, \tup a)$. If $\Delta_{\intern}$ is a collection of internal formulas and
\be\label{antecedn}
\P + \Delta_{\intern} \vdash \Phi(\tup a), 
\ee
then one can extract from the proof a sequence of closed terms $t$ in $\mathcal{T}^{*}$ such that
\be\label{consequalty}
\textup{\textsf{E-PA}}^{\omega*} + \Delta_{\intern} \vdash\  \forall \tup x \, \exists \tup y\in \tup t(\tup x)\ \varphi(\tup x,\tup y, \tup a).
\ee
\end{thm}
\begin{proof}
Immediate by \cite{brie}*{Theorem 7.7}.  
\end{proof}
%The proofs of the soundness theorems in \cite{brie}*{\S5-7} provide an algorithm $\mathcal{A}$ to obtain the term $t$ from the theorem.  
The following corollary is essential to our results.  We shall refer to a formula of the form $(\forall^{\st}\underline{x})(\exists^{\st}\underline{y})\psi(\underline{x},\underline{y}, \underline{a})$, where $\varphi$ is internal, as a \emph{normal form}.   
\begin{cor}\label{consresultcor}
If for internal $\psi$ the formula $\Phi(\underline{a})\equiv(\forall^{\st}\underline{x})(\exists^{\st}\underline{y})\psi(\underline{x},\underline{y}, \underline{a})$ satisfies \eqref{antecedn}, then 
$(\forall \underline{x})(\exists \underline{y}\in t(\underline{x}))\psi(\underline{x},\underline{y},\underline{a})$ is proved in the corresponding \eqref{consequalty}.  
\end{cor}
\begin{proof}
Clearly, if for $\psi$ and $\Phi$ as given we have $\Phi(\underline{a})^{S_{\st}}\equiv \Phi(\underline{a})$, then the corollary follows immediately from the theorem.  
A tedious but straightforward verification using the clauses (i)-(v) in \cite{brie}*{Def.\ 7.1} establishes that indeed $\Phi(\underline{a})^{S_{\st}}\equiv \Phi(\underline{a})$.  For completeness, this verification is performed in Section \ref{apie}.    \qed
\end{proof}
Finally, the previous theorems do not really depend on the presence of full Peano arithmetic.  
Indeed, let \textsf{E-PRA}$^{\omega}$ be the system defined in \cite{kohlenbach2}*{\S2} and let \textsf{E-PRA}$^{\omega*}$ 
be its definitional extension with types for finite sequences as in \cite{brie}*{\S2}.  We permit ourselves a slight abuse of notation by not distinguishing between Kohlenbach's $\RCAo\equiv \textup{\textsf{E-PRA}}^{\omega}+\QFAC^{1,0}$ (See \cite{kohlenbach2}*{\S2}) and $\textup{\textsf{E-PRA}}^{\omega*}+\QFAC^{1,0}$.  
\begin{cor}\label{consresultcor2}
The previous theorem and corollary go through for $\P$ and $\textup{\textsf{E-PA}}^{\omega*}$ replaced by $\RCAO\equiv \textsf{\textup{E-PRA}}^{\omega*}+\T_{\st}^{*} +\HAC_{\INT} +\textsf{\textup{I}}+\QFAC^{1,0}$ and $\RCAo$.  
\end{cor}
\begin{proof}
The proof of \cite{brie}*{Theorem 7.7} goes through for any fragment of \textsf{E-PA}$^{\omega{*}}$ which includes \textsf{EFA}, sometimes also called $\textsf{I}\Delta_{0}+\textsf{EXP}$.  
In particular, the exponential function is (all what is) required to `easily' manipulate finite sequences.    
\end{proof}

\subsection{Notations}
We mostly use the same notations as in \cite{brie}, some of which we repeat here.  
\begin{rem}[Notations]\label{notawin}\rm
We write $(\forall^{\st}x^{\tau})\Phi(x^{\tau})$ and $(\exists^{\st}x^{\sigma})\Psi(x^{\sigma})$ as short for 
$(\forall x^{\tau})\big[\st(x^{\tau})\di \Phi(x^{\tau})\big]$ and $(\exists x^{\sigma})\big[\st(x^{\sigma})\wedge \Psi(x^{\sigma})\big]$.     
We also write $(\forall x^{0}\in \Omega)\Phi(x^{0})$ and $(\exists x^{0}\in \Omega)\Psi(x^{0})$ as short for 
$(\forall x^{0})\big[\neg\st(x^{0})\di \Phi(x^{0})\big]$ and $(\exists x^{0})\big[\neg\st(x^{0})\wedge \Psi(x^{0})\big]$.  Furthermore, if $\neg\st(x^{0})$ (resp.\ $\st(x^{0})$), we also say that $x^{0}$ is `infinite' (resp.\ finite) and write `$x^{0}\in \Omega$'.  %In keeping with the internal 
Finally, a formula $A$ is `internal' if it does not involve $\st$, and $A^{\st}$ is defined from $A$ by appending `st' to all quantifiers (except bounded number quantifiers).    
\end{rem}
Secondly, we will use the usual notations for rational and real numbers and functions as introduced in \cite{kohlenbach2}*{p.\ 288-289} (and \cite{simpson2}*{I.8.1} for the former).  
\begin{defi}[Real number]\label{keepintireal}\rm
A (standard) real number $x$ is a (standard) fast-converging Cauchy sequence $q_{(\cdot)}^{1}$, i.e.\ $(\forall n^{0}, i^{0})(|q_{n}-q_{n+i})|<_{0} \frac{1}{2^{n}})$.  
We freely make use of Kohlenbach's `hat function' from \cite{kohlenbach2}*{p.\ 289} to guarantee that every sequence $f^{1}$ can be viewed as a real.  We also use the notation $[x](k):=q_{k}$ for the $k$-th approximation of real numbers.    
Two reals $x, y$ represented by $q_{(\cdot)}$ and $r_{(\cdot)}$ are \emph{equal}, denoted $x=_{\R}y$, if $(\forall n)(|q_{n}-r_{n}|\leq \frac{1}{2^{n}})$. Inequality $<_{\R}$ is defined similarly.         
We also write $x\approx y$ if $(\forall^{\st} n)(|q_{n}-r_{n}|\leq \frac{1}{2^{n}})$ and $x\gg y$ if $x>_{\R}y\wedge x\not\approx y$.  Functions $F:\R\di \R$ are represented by $\Phi^{1\di 1}$ such that 
\be\tag{\textsf{RE}}
(\forall x, y)(x=_{\R}y\di \Phi(x)=_{\R}\Phi(y)),
\ee
 i.e.\ equal reals are mapped to equal reals.  Finally, sets are denoted $X^{1}, Y^{1}, Z^{1}, \dots$ and are given by their characteristic functions $f^{1}_{X}$, i.e.\ $(\forall x^{0})[x\in X\asa f_{X}(x)=1]$, where $f_{X}^{1}$ is assumed to be binary.      
\end{defi}
Thirdly, we use the usual extensional notion of equality.  
%None of our results depend crucially on this notion, i.e.\ we would obtain the same results for non-extensional notions of equality by dropping the axiom \eqref{EXT} introduced now.    
\begin{rem}[Equality]\label{equ}\rm
Equality between natural numbers `$=_{0}$' is a primitive.  Equality `$=_{\tau}$' for type $\tau$-objects $x,y$ is then defined as follows:
\be\label{aparth}
[x=_{\tau}y] \equiv (\forall z_{1}^{\tau_{1}}\dots z_{k}^{\tau_{k}})[xz_{1}\dots z_{k}=_{0}yz_{1}\dots z_{k}]
\ee
if the type $\tau$ is composed as $\tau\equiv(\tau_{1}\di \dots\di \tau_{k}\di 0)$.
In the spirit of Nonstandard Analysis, we define `approximate equality $\approx_{\tau}$' as follows:
\be\label{aparth2}
[x\approx_{\tau}y] \equiv (\forall^{\st} z_{1}^{\tau_{1}}\dots z_{k}^{\tau_{k}})[xz_{1}\dots z_{k}=_{0}yz_{1}\dots z_{k}]
\ee
with the type $\tau$ as above.  
%As noted in \cite{brie}*{p.\ 1973}, weak nonstandard systems can include the axiom of extensionality for $=_{\tau}$, but not for $\approx_{\tau}$.
The system $\P$ includes the \emph{axiom of extensionality}: %for all $\varphi^{\rho\di \tau}$ as follows:
\be\label{EXT}\tag{\textsf{E}}  
(\forall  x^{\rho},y^{\rho}, \varphi^{\rho\di \tau}) \big[x=_{\rho} y \di \varphi(x)=_{\tau}\varphi(y)   \big].
\ee
However, as noted in \cite{brie}*{p.\ 1973}, the so-called axiom of \emph{standard} extensionality \eqref{EXT}$^{\st}$ is problematic and cannot be included in $\P$.  
Finally, a functional $\Xi^{ 1\di 0}$ is called an \emph{extensionality functional} for $\varphi^{1\di 1}$ if 
\be\label{turki}
(\forall k^{0}, f^{1}, g^{1})\big[ \overline{f}\Xi(f,g, k)=_{0}\overline{g}\Xi(f,g,k) \di \overline{\varphi(f)}k=_{0}\overline{\varphi(g)}k \big],  
\ee
i.e.\ $\Xi$ witnesses \eqref{EXT} for $\varphi$.  
As will become clear in Section \ref{main}, $\eqref{EXT}^{\st}$ is translated to the existence of an extensionality functional when applying Corollary \ref{consresultcor}.  
\end{rem} 
\section{Main results}\label{main}
In this section, we prove that from certain RM-equivalences in Nonstandard Analysis, one can extract \emph{effective} RM-equivalences in Kohlenbach's higher-order RM.   %thanks to the framework \cite{brie} sketched in Section \ref{PIPI}.  
The notion of `effective' implication is defined as expected.  
\bdefi[Effective implication]\label{effimp}
An implication $(\exists \Phi)A(\Phi)\di (\exists \Psi)B(\Psi)$ is \emph{effective} if there is a term $t$ in the language such that additionally $(\forall \Phi)[A(\Phi)\di B(t(\Phi))]$, i.e.\ $\Psi$ can be effectively defined in terms of $\Phi$.    
\edefi
The terms involved in our effective implications are \emph{primitive recursive} in the sense of G\"odel's system \textsf{T}, as discussed in Section \ref{PIPI}.  
In light of the elementary nature of an extensionality functional (See Remark~\ref{equ}), we still refer to an implication as `effective', if the term $t$ as in Definition \ref{effimp} involves an extensionality functional.  This assumption is not entirely innocent by \cite{kooltje}*{Remark 3.6}.    

\medskip

We shall treat the uniform version of weak K\"onig's lemma in Section \ref{FUWKL} in some detail, by way of illustration.  
In Section \ref{group1}, we treat the main theorem from \cite{simha} pertaining to group theory, using the proofs from Section \ref{FUWKL} as a template.  
Similar results for the RM of weak K\"onig's lemma may be proved similarly.    
\subsection{Uniform weak K\"onig's lemma}\label{FUWKL}
In this section, we first establish a particular nonstandard equivalence involving a fragment of Nelson's axiom \emph{Transfer}, and a (nonstandard and uniform) version of weak K\"onig's lemma.  As a result of applying Corollary~\ref{consresultcor2} to this nonstandard equivalence, we obtain an \emph{effective} equivalence between the uniform version of weak K\"onig's lemma $\UWKL$ and a version of arithmetical comprehension $(\mu^{2})$.  % which are defined as follows: 
\be\tag{$\UWKL$}
(\exists \Phi^{1\di 1})(\forall T \leq_{1}1)\big[ (\forall n)(\exists \alpha)(|\alpha|=n\wedge \alpha\in T)\di (\forall m)(\overline{\Phi(T)}m\in T) \big]
\ee
\be\tag{$\mu^{2}$}\label{Frak}
(\exists \mu^{2})(\forall f^{1})\big[(\exists x^{0})f(x)=0 \di  f(\mu(f)) \big]
\ee  
The functional $(\mu^{2})$ is also known as \emph{Feferman's non-constructive mu-operator} (See \cite{avi2}).  
We also need the following restriction of Nelson's axiom \emph{Transfer}:
 \be\tag{$\paai$}
(\forall^{\st}f^{1})\big[(\forall^{\st}n^{0})f(n)=0\di (\forall m)f(m)=0\big].
\ee  
Define $\UWKL(\Phi, T)$ as $\UWKL$ without the leading quantifiers, and $\UWKL^{+}$ as:
\[
(\exists^{\st}\Phi^{1\di 1})\big[(\forall^{\st}T^{1})\UWKL(\Phi, T)\wedge (\forall^{\st} T^{1}, S^{1})\big(T\approx_{1} S \di \Phi(T)\approx_{1}\Phi(S) \big)\big],  
\]
Note that the second conjunct expresses that $\Phi$ is \emph{standard extensional} (See Remark~\ref{equ});  
Finally, let $\MU(\mu)$ be $(\mu^{2})$ without the leading existential quantifier.  
\begin{thm}\label{proto}
From a proof in $\RCAO$ that $\UWKL^{+}\asa \paai$, terms $s,t$ can be extracted such that $\RCAo$ proves:
\be\label{frood8}
(\forall \mu^{2})\big[\textsf{\MU}(\mu)\di \UWKL(s(\mu)) \big] \wedge (\forall \Phi^{1\di 1})\big[ \UWKL(\Phi)\di  \MU(t(\Phi, \Xi))  \big].
\ee
where $\Xi$ is an extensionality functional for $\Phi$.  
\end{thm}
\begin{proof}
Due to space constraints, we shall only prove that $\UWKL^{+}\di \paai$ in $\RCAO$, and obtain the associated second conjunct of \eqref{frood8}, which is the most interesting result anyway.  The remaining results are proved in Section~\ref{kulll}.  

\medskip

To prove $\UWKL^{+}\di \paai$ in $\RCAO$, assume $\UWKL^{+}$ and suppose $\paai$ is false, i.e.\ there is $f$ such that $(\forall^{\st}n)f(n)=0 \wedge (\exists m^{0})f(m)\ne 0$.
Now define the tree $T_{i}$ for $i=0,1$ as follows
\[
\sigma\in T_{i}\asa \big[(\forall n^{0}<|\sigma|)(\sigma(n)=i) \vee \big[(\forall n^{0}<|\sigma|)(\sigma(n)=1-i)\wedge (\forall m\leq |\sigma|)f(n)=0\big] \big].
\]
Note that $T_{0}\approx_{1} T_{1}$ but that the former (resp.\ the latter) only has one path, namely $00\dots$ (resp.\ $11\dots$).  Hence, we must have 
$\Phi(T_{0})\not\approx_{1} \Phi(T_{1})$ for $\Phi$ as in $\UWKL^{+}$, which however contradicts the standard extensionality of $\Phi$.  
In light of this contradiction, we have $\UWKL^{+}\di \paai$.  

\medskip

Finally, we shall prove the second conjunct in \eqref{frood8}.  %The first conjunct is proved in exactly the same way.  
Note that $\paai$ can be brought into
the following normal form: $(\forall^{\st}f^{1})(\exists^{\st}i^{0})\big[(\exists n^{0})f(n)=0\di (\exists m\leq i)f(m)=0\big]$, where the formula in square brackets is abbreviated by $B(f, i)$.  Similarly, the second conjunct of $\UWKL^{+}$ has the following normal form:  
\be\label{kurve}
(\forall^{\st} T^{1}, S^{1}, k^{0})(\exists^{\st} N)\big(\overline{T}N=_{0}\overline{S}N \di \overline{\Phi(T)}k=_{0}\overline{\Phi(S)}k\big),
\ee
which is immediate by resolving `$\approx_{1}$'; We denote the formula in square brackets by $A(T, S, N, k, \Phi)$.  
Hence, $\UWKL^+\di \paai$ is easily seen to imply:
\[
(\forall^{\st}\Phi, \Xi)\big[ [(\forall^{\st}T^{1})\UWKL(\Phi, T)\wedge (\forall^{\st} U^{1}, S^{1}, k^{0})A(U, S, \Xi(U, S, k), k, \Phi)]  \di     (\forall^{\st}f^{1})(\exists^{\st}n)B(f,n)\big] 
\]
Dropping the `st' in the antecedent of the implication and brining out the remaining standard quantifiers, we obtain
\[
(\forall^{\st}\Phi, \Xi ,f^{1})(\exists^{\st}n)\big[ [(\forall T^{1})\UWKL(\Phi, T)\wedge (\forall  U^{1}, S^{1}, k^{0})A(U, S, \Xi(U, S, k), k, \Phi)]  \di    B(f,n)\big] 
\]
Let $C(\Phi, \Xi, f, n)$ be the formula in big square brackets and apply Corollary~\ref{consresultcor2} to `$\RCAO\vdash (\forall^{\st}\Phi, \Xi, f^{1})(\exists^{\st}n)C(\Phi, \Xi, f, n)$' to obtain a term $t$ such that $\RCAo$ proves $ (\forall \Phi, \Xi, f^{1})(\exists n\in t(\Phi, \Xi, f))C(\Phi, \Xi, f, n)$.  Now define the term $s(\Phi, \Xi, f)$ as $\max_{i<|t(\Phi, \Xi, f)|}t(\Phi, \Xi, f)(i)$ and note that $s$ provides the functional $(\mu^{2})$ if $\Phi$ satisfies $(\forall T^{1})\UWKL(\Phi,T )$ and $\Xi$ is the associated extensionality functional. 
\qed
\end{proof}
Thus, we have obtained the effective equivalence $\UWKL\asa (\mu^{2})$ from the associated nonstandard equivalence $\paai\asa \UWKL^{+}$.
Now, Kohlenbach proves the equivalence $(\mu^{2})\asa \UWKL$ in \cite{kooltje}, and also established a related \emph{effective} equivalence.  
Hence, the final equivalence in Theorem~\ref{proto} is not that surprising, but our methodology arguably is:  The \emph{effective} equivalence in the latter theorem is obtained \emph{automatically} (in the sense that the terms $s, t$ can be `read off' from the nonstandard proof in $\RCAO$) from a proof in which no attention to computational content is given, and Nonstandard Analysis is even used.  

\medskip

In the aforementioned proof from \cite{kooltje}, $\Phi$ from $\UWKL$ is shown to be discontinuous, and \emph{Grilliot's trick} is then applied to obtain the Halting problem.  Intuitively speaking, the proof of Theorem \ref{proto} is similar: $\Phi$ from $\UWKL^{+}$ is 
\emph{nonstandard {dis}continuous} in light of $T_{0}\approx_{1}T_{1}\wedge \Phi(T_{0})\not\approx_{1}\Phi(T_{1})$, and the latter property allows us to derive a contradiction from the combination of standard extensionality and the negation of $\paai$.  The Transfer principle $\paai$ becomes $(\mu^{2})$ (which also solves the Halting problem) after applying Corollary~\ref{consresultcor2}.

\subsection{Group theory and order}\label{group1}
In this section, we prove a nonstandard equivalence for Levi's theorem for countable abelian groups from \cite{simha} and, as in the previous section, extract an effective equivalence.  We also study the \emph{contraposition} of Levi's theorem in the same way, yielding rather different results.  Group theory is introduced in RM in \cite{simpson2}*{III.6}.   
\begin{defi}[\cite{simha}*{\S2}]\label{shiova}
Let $A$ be a countable abelian group. $A$ is torsion-free if $(\forall n^{0})(\forall a\in A\setminus\{0_{A}\})(n\times a\ne 0)$. $A$ is orderable if there exists a linear
ordering `$<$' on $A$ such that $(\forall a, b, c\in A)(a\leq b \di a+c \leq b+c)$. $P$ is the positive cone of $A$ if $P=\{ a\in A: a\geq 0\}$.  
\edefi
Let $\textsf{ORD}$ be Levi's theorem that every torsion-free countable abelian group is orderable.  Define $\UORD(\Phi, A)$ as the statement that $\Phi(A)$ is the order for $A$ as in $\textsf{ORD}$.  Finally, define $\UORD^{+}$ as follows:
\[
(\exists^{\st}\Phi^{1\di 1})\big[(\forall^{\st}A^{1})\UORD(\Phi, A)\wedge (\forall^{\st} A^{1}, B^{1})\big(A\approx_{1} B \di \Phi(A)\approx_{1}\Phi(B) \big)\big].  
\]
We have the following theorem.
\begin{thm}\label{proto2}
From a proof in $\RCAO$ that $\UORD^{+}\asa \paai$, terms $s,t$ can be extracted such that $\RCAo$ proves:
\be\label{frood9}
(\forall \mu^{2})\big[\textsf{\MU}(\mu)\di \UORD(s(\mu)) \big] \wedge (\forall \Phi^{1\di 1})\big[ \UORD(\Phi)\di  \MU(t(\Phi, \Xi))  \big].
\ee
where $\Xi$ is an extensionality functional for $\Phi$.  
\end{thm}
\begin{proof}  
Given the similarity in syntactic structure between $\ORD$ and $\WKL$, the proof of $\paai\di \UORD^{+}$ in $\RCAO$ follows immediately from the proof of $\WKL \di \textsf{ORD}$  (\cite{simha}*{Lemma 2.3}) and the proof of Theorem \ref{proto} in Section \ref{kulll}.  The proof that $\UORD^{+}\di \paai$ proceeds as follows.  
Note that the ordering provided by $\ORD$ is not unique:  If `$\leq$' is an order on $A$, then `$\sqsubseteq$' obtained by $a\sqsubset b \equiv \neg(a\leq b)$ is also one, and the associated positive cones only intersect in the neutral element $0_{A}$.     
Now assume $\UORD^{+}$ and suppose that $\paai$ is false, i.e.\ there is standard $f^{1}$ such that $(\forall^{\st}n^{0})f(n)=0$ and $(\exists m^{0})f(m)\ne 0$.  Take a standard torsion-free abelian group $A^{1}=(a_{0}, a_{2}, a_{3}, \dots)$, and define the following standard group:  
\be\label{fruj}
B(i):=
\begin{cases}
-A(i) & (\exists n\leq \max\{ A(i), -A(i)\})f(n)\ne 0\\
~~A(i)& \text{otherwise},
\end{cases}
\ee
Note that the `inverse' operation `$-$' associated to $A$ is also standard, implying that $-a$ is standard if and only if $a\in A$ is.  
As a result, the modification via $f$ to $A$ in \eqref{fruj} does not change the standard part of $A$, i.e.\ we have $A\approx_{1}B$.  
However, we also have $\Phi(A)\not\approx_{1}\Phi(B)$, as \eqref{fruj} switches $A(i)$ and $-A(i)$ in the enumeration of $A$ for large enough $i$.  
In particular, the positive cone of $A$ only intersects the positive cone of $B$ in $0_{A}=0_{B}$ due to the `switch' taking place in the first clause of \eqref{fruj}.  
The previous contradiction yields $\UORD^{+}\di \paai$.  

\medskip

Finally, in light of the similarity in syntactic structure between $\UORD^{+}$ and $\UWKL^{+}$, one obtains \eqref{frood9} from $\UORD^{+}\asa\paai$ in the same way as one obtains \eqref{frood8} from $\UWKL^{+}\asa \paai$, and we are done.
\qed
\end{proof}  
Thus, \eqref{frood9} establishes the effective equivalence between $\WKL$ and $\ORD$, and we now study the effective equivalence between 
the \emph{contrapositions} of the latter, leading to \emph{quite} different results.  
Recall that the fan theorem, denoted $\FAN$, is the classical contraposition of weak K\"onig's lemma.  Similarly, let $\textsf{DRO}$ be the contraposition 
of $\ORD$ and consider the following explicit versions:
\begin{align}\label{UFAN}\tag{$\UFAN(\Phi)$}
(\forall T^{1}\leq_{1}1,g^{2})\big[(\forall \beta\leq_{1}1)&\overline{\beta}g(\beta)\not \in T \\
&\di (\forall \beta \leq_{1}1)(\exists i\leq \Phi(g))\overline{\beta}i\not\in T   \big].\notag
\end{align}
\begin{align}
(\forall A^{1}, h^{2})\big[
 (\forall X^{1})(\exists a, b, &c \in A)(a, b, c \leq h(X) \wedge a+c >_{X} b+c) \tag{$\UDRO(\Psi)$} \\
&   \di(\exists n^{0}, a\in A)( n,a \leq \Psi(h)\wedge n\times a =0_{A})  )\big]. \notag
\end{align}
We have the following theorem.
\begin{thm}\label{soareyouuuuuu}
From the proof of $\WKL\asa \ORD$ in $\RCA_{0}$ \(See \cite{simha}*{p.\ 179}\), terms $s, t$ can be extracted witnessing the \emph{explicit} equivalence $\FAN\asa \ORD$ in $\RCAo$:
\be\label{UKI}
(\forall \Phi^{3})[\UFAN(\Phi)\di \textsf{\textup{UDRO}}(s(\Phi))] \wedge (\forall \Psi^{3})[\textup{\textsf{UDRO}}(\Psi)\di \textsf{\textup{UFAN}}(t(\Psi))].  
\ee
\end{thm}
\begin{proof}
First of all, assume the following two theorems are equivalent in $\RCAO$.
\begin{align}
(\forall^{\st} T^{1}\leq_{1}1,g^{2})\big[(\forall \beta\leq_{1}1)&\overline{\beta}g(\beta)\not \in T\label{kral} \\
&\di (\exists^{\st}k)(\forall \beta \leq_{1}1)(\exists i\leq k)\overline{\beta}i\not\in T   \big].\notag
\end{align}
\begin{align}
(\forall^{\st} A^{1}, h^{2})\big[
 (\forall X^{1})(\exists a, b, &c \in A)(a, b, c \leq h(X) \wedge a+c >_{X} b+c)\label{kral2} \\
&   \di (\exists^{\st}m)(\exists n^{0}, a\in A)( n,a \leq m\wedge n\times a =0_{A})  )\big]. \notag
\end{align}
Applying Corollary \ref{consresultcor2} to the normal forms of `\eqref{kral}$\di$\eqref{kral2}' and `\eqref{kral2}$\di$\eqref{kral}' now immediately yields \eqref{UKI}.
It is a straightforward verification that the proof that $\ORD\asa \WKL$ in \cite{simha}*{Theorem 2.5} can be (easily) modified to a proof of `$\eqref{kral}\asa \eqref{kral2}$' in $\RCAO$.  For completeness, we prove one direction of the latter implication in Section \ref{fappendix}.  
\qed
\end{proof}
The principle $(\exists \Phi)\UFAN(\Phi)$ is conservative over weak K\"onig's lemma by \cite{kohlenbach2}*{Prop.~3.15}, while $(\mu^{2})$ is essentially arithmetical comprehension.  Hence, there is a big difference in strength between $\UORD$ and $(\exists \Psi)\UDRO(\Psi)$, which can be explained by the different constructive status of $\ORD$ and $\textsf{DRO}$ (See \cite{kohlenbach2}*{p.\ 294}).   

\medskip

Finally, we show that from certain effective implications, one can re-obtain the original nonstandard theorem.  % from which they were obtained.  
To this end, consider:  % the following implication:
\be\label{pm}\tag{\textsf{HER}$(i,o)$}
(\forall \Phi , \Xi, f)\big[ \UWKL_{\pw}(\Phi, i) \wedge \textsf{\textup{EXT}}_{\pw}(\Phi, \Xi, i)\di  \MU_{\pw}(o(\Phi, \Xi,f),f)
\ee
where $ \MU_{\pw}(\mu,f)$ is $\MU(\mu)$ with the leading universal quantifier dropped, where $\UWKL_{\pw}(\Phi, i)$  is $(\forall T^{1}\in i(\Phi, \Xi, f)(1))\UWKL(\Phi, T)$, and where $\textsf{\textup{EXT}}_{\pw}(\Phi, \Xi, i)$ is 
\[
(\forall U, S, k \in i(\Phi, \Xi, f)(2))(\overline{S}\Xi(U, S, k)=\overline{U}\Xi(U, S, k) \di \overline{\Phi(U)}k=\overline{\Phi(S)}k.
\]
Intuitively speaking \textsf{HER}$(i,o)$ expresses that in order to compute Feferman's mu-operator via $o$ \emph{for one single $f^{1}$}, 
one needs to supply $\Phi$ which satisfies $\UWKL$ \emph{but only for the finite sequence of trees given by $i$}, and similar for $\Xi$.   
In other words, \textsf{HER}$(i,o)$ is a `pointwise' version of the second conjunct of \eqref{frood8}.  
\begin{thm}
From $\RCAO\vdash\UWKL^{+}\di \paai$, closed terms $i, o$ can be extracted such that $\RCAo\vdash \textup{\textsf{HER}}(i, o)$.  
If for closed terms $i, o$, $\RCAo$ proves $\textup{\textsf{HER}}(i, o)$, then the latter proof yields a proof of $\UWKL^{+}\di \paai$ in $\RCAO$.    
\end{thm}
\begin{proof}
For the first part of the theorem, consider the proof of Theorem \ref{proto}, in particular the sentence below \eqref{kurve} as follows:
\[
(\forall^{\st}\Phi, \Xi)\big[ [(\forall^{\st}T^{1})\UWKL(\Phi, T)\wedge (\forall^{\st} U^{1}, S^{1}, k^{0})A(U, S, \Xi(U, S, k), k)] 
 \di     (\forall^{\st}f^{1})(\exists^{\st}n)B(f,n)\big]. 
\]
Bringing the previous into normal form, \emph{without} dropping `st' predicates, yields
\[
(\forall^{\st}\Phi, \Xi, f^{1})(\exists^{\st} T, U, S, k, n )\big[ [\UWKL(\Phi, T)\wedge A(U, S, \Xi(U, S, k), k)]  \di   B(f,n)\big], 
\]
where $D(\Phi, \Xi, f, T, U, S, k, n)$ abbreviates the formula in big square brackets.  
Applying Corollary \ref{consresultcor2} to the fact that $\RCAO$ proves 
\[
(\forall^{\st}\Phi, \Xi, f^{1})(\exists^{\st} T, U, S, k, n )D(\Phi, \Xi, f, T, U, S, k, n)
\]
yields a term $t$ such that $\RCAo$ proves:
\be\label{forgot}
(\forall \Phi, \Xi, f^{1})(\exists  T, U, S, k, n \in t(\Phi, \Xi, f))D(\Phi, \Xi, f, T, U, S, k, n).
\ee
Now define the term $o$ as the maximum of all entries of $t$ pertaining to $n$, and define $i(\Phi, \Xi, f)(1)$ as the sequence of all entries of $t$ pertaining to $T$, and the same for  $i(\Phi, \Xi, f)(2)$ and $U, S, k$.  With these definitions, \eqref{forgot} implies $\textsf{HER}(i,o)$.  

\medskip

For the second part of the proof, by the second standardness axiom from Definition \ref{debs}, terms like $i, o$ are standard in $\RCAO$.  
Hence, if $\RCAo\vdash\textsf{HER}(i,o)$, then $\RCAO$ proves $\textsf{HER}(i,o)\wedge \st(i)\wedge \st(o)$.  
Thus, for standard $\Phi, \Xi, f$, $i(\Phi, \Xi, f)$ and $o(\Phi, \Xi, f)$ are standard, yielding that $\MU_{\pw}(o(\Phi, \Xi,f),f)$ for standard inputs implies 
$(\exists n)f(n)= 0\di (\exists^{\st}m)f(m)=0$.  Hence, $\textsf{HER}(i,o)$ yields     
\[
(\forall^{\st} \Phi , \Xi, f)\big[ \UWKL_{\pw}(\Phi, i) \wedge \textsf{\textup{EXT}}_{\pw}(\Phi, \Xi, i)\di [(\exists n)f(n)= 0\di (\exists^{\st}m)f(m)=0]\big],
\]
and bring two of the standard universal quantifiers inside as follows:
\[
(\forall^{\st} f)\big[(\exists^{\st}\Phi, \Xi) [\UWKL_{\pw}(\Phi, i) \wedge \textsf{\textup{EXT}}_{\pw}(\Phi, \Xi, i)]\di [(\exists n)f(n)= 0\di (\exists^{\st}m)f(m)=0]\big],
\]
and note that $\UWKL^{+}$ implies the antecedent of the previous, i.e.\ 
\[
(\forall^{\st}  f)\big[\UWKL^{+}\di [(\exists n)f(n)= 0\di (\exists^{\st}m)f(m)=0]\big],
\]
which immediately yields $\UWKL^{+}\di \paai$.  
\qed
\end{proof}
We refer to $\textup{\textsf{HER}}(i, o)$ as the \emph{Herbrandisation} of the implication $\UWKL^{+}\di \paai$.  
As suggested by the notation, the Herbrandisation is a `pointwise' version of the second conjunct of \eqref{frood8}.  
We could obtain Hebrandisations of e.g.\ \eqref{frood9}, but do not go into details due to space constraints.  
\section{Bibliography}
%\begin{bibdiv}

%\end{bibdiv}

%\newpage
\appendix

\section{Technical Appendix}

\subsection{Proof of Corollary \ref{consresultcor}}\label{apie}
In this section, we prove Corollary \ref{consresultcor}, as follows.
\begin{cor}\label{consresultcor44}
If for internal $\psi$, $\Phi(\underline{a})\equiv(\forall^{\st}\underline{x})(\exists^{\st}\underline{y})\psi(\underline{x},\underline{y}, \underline{a})$ satisfies \eqref{antecedn}, then 
$(\forall \underline{x})(\exists \underline{y}\in t(\underline{x}))\psi(\underline{x},\underline{y},\underline{a})$ is proved in the corresponding formula \eqref{consequalty}.  
\end{cor}
\begin{proof}
Clearly, if for $\psi$ and $\Phi$ as given we have $\Phi(\underline{a})^{S_{\st}}\equiv \Phi(\underline{a})$, then the corollary follows immediately from the theorem.  
A tedious but straightforward verification using the clauses (i)-(v) in \cite{brie}*{Def.\ 7.1} establishes that indeed $\Phi(\underline{a})^{S_{\st}}\equiv \Phi(\underline{a})$.  
For completeness, we now list these five inductive clauses and perform this verification.  

\medskip

Hence, if $\Phi(\underline{a})$ and $\Psi(\underline{b})$  in the language of $\P$ have the following interpretations
\be\label{dombu}
\Phi(\underline{a})^{S_{\st}}\equiv (\forall^{\st}\underline{x})(\exists^{\st}\underline{y})\varphi(\underline{x},\underline{y},\underline{a}) \textup{ and } \Psi(\underline{b})^{S_{\st}}\equiv (\forall^{\st}\underline{u})(\exists^{\st}\underline{v})\psi(\underline{u},\underline{v},\underline{b}),
\ee
then they interact as follows with the logical connectives by \cite{brie}*{Def.\ 7.1}:
\begin{enumerate}[(i)]
\item $\psi^{S_{\st}}:=\psi$ for atomic internal $\psi$.  
\item$ \big(\st(z)\big)^{S_{\st}}:=(\exists^{\st}x)(z=x)$.
\item $(\neg \Phi)^{S_{\st}}:=(\forall^{\st} \underline{Y})(\exists^{\st}\underline{x})(\forall \underline{y}\in \underline{Y}[\underline{x}])\neg\varphi(\underline{x},\underline{y},\underline{a})$.  
\item$(\Phi\vee \Psi)^{S_{\st}}:=(\forall^{\st}\underline{x},\underline{u})(\exists^{\st}\underline{y}, \underline{v})[\varphi(\underline{x},\underline{y},\underline{a})\vee \psi(\underline{u},\underline{v},\underline{b})]$
\item $\big( (\forall z)\Phi \big)^{S_{\st}}:=(\forall^{\st}\underline{x})(\exists^{\st}\underline{y})(\forall z)(\exists \underline{y}'\in \underline{y})\varphi(\underline{x},\underline{y}',z)$
\end{enumerate}
Hence, fix $\Phi_{0}(\underline{a})\equiv(\forall^{\st}\underline{x})(\exists^{\st}\underline{y})\psi_{0}(\underline{x},\underline{y}, \underline{a})$ with internal $\psi_{0}$, and note that $\phi^{S_{\st}}\equiv\phi$ for any internal formula.  
We have $[\st(\underline{y})]^{S_{\st}}\equiv (\exists^{\st} \underline{w})(\underline{w}=\underline{y})$ and also 
\[
[\neg\st(\underline{y})]^{S_{\st}}\equiv (\forall^{\st} \underline{W} ) (\exists^{\st}\underline{x})(\forall \underline{w}\in \underline{W}[\underline{x}])\neg(\underline{w}=\underline{y})\equiv (\forall^{\st}\underline{w})(\underline{w}\ne \underline{y}).  
\]    
Hence, $[\neg\st(\underline{y})\vee\neg \psi_{0}(\underline{x}, \underline{y}, \underline{a})]^{S_{\st}}$ is just $(\forall^{\st}\underline{w})[(\underline{w}\ne \underline{y}) \vee \neg \psi_{0}(\underline{x}, \underline{y}, \underline{a})]$, and 
\[
\big[(\forall \underline{y})[\neg\st(\underline{y})\vee \neg\psi_{0}(\underline{x}, \underline{y}, \underline{a})]\big]^{S_{\st}}\equiv
 (\forall^{\st}\underline{w})(\exists^{\st}\underline{v})(\forall \underline{y})(\exists \underline{v}'\in \underline{v})[\underline{w}\ne\underline{y}\vee \neg\psi_{0}(\underline{x}, \underline{y}, \underline{a})].
\]
which is just $(\forall^{\st}\underline{w})(\forall \underline{y})[(\underline{w}\ne \underline{y}) \vee \neg\psi_{0}(\underline{x}, \underline{y}, \underline{a})]$.  Furthermore, we have
\begin{align*}
\big[(\exists^{\st}y)\psi_{0}(\underline{x}, \tup y, \tup a)\big]^{S_{\st}}&\equiv\big[\neg(\forall \underline{y})[\neg\st(\underline{y})\vee\neg \psi_{0}(\underline{x}, \underline{y}, \underline{a})]\big]^{S_{\st}}\\
&\equiv(\forall^{\st} \underline{V})(\exists^{\st}\underline{w})(\forall \underline{v}\in \underline{V}[\underline{w}])\neg[(\forall \underline{y})[(\underline{w}\ne \underline{y}) \vee \neg\psi_{0}(\underline{x}, \underline{y}, \underline{a})]].\\
&\equiv (\exists^{\st}\underline{w})(\exists \underline{y})[(\underline{w}= \underline{y}) \wedge \psi_{0}(\underline{x}, \underline{y}, \underline{a})]]\equiv (\exists^{\st}\underline{w})\psi_{0}(\underline{x}, \underline{w}, \underline{a}).
\end{align*}
Hence, we have proved so far that $(\exists^{\st}\underline{y})\psi_{0}(\underline{x}, \underline{y}, \underline{a})$ is invariant under $S_{\st}$.  By the previous, we also obtain:  
\[
\big[\neg \st(\underline{x})\vee (\exists^{\st}y)\psi_{0}(\underline{x}, \tup y, \tup a)\big]^{S_{\st}}\equiv  (\forall^{\st}\underline{w}')(\exists^{\st} \underline{w})[(\underline{w}'\ne \underline{x}) \vee \psi_{0}(\tup x, \tup w, \tup a)].
\]
Our final computation now yields the desired result: 
\begin{align*}
\big[(\forall^{\st} \underline{x})(\exists^{\st}y)\psi_{0}(\underline{x}, \tup y, \tup a)\big]^{S_{\st}}
&\equiv\big[(\forall \underline{x})(\neg \st(\underline{x})\vee (\exists^{\st}y)\psi_{0}(\underline{x}, \tup y, \tup a))\big]^{S_{\st}}\\
&\equiv(\forall^{\st}\underline{w}')(\exists^{\st} \underline{w})(\forall \underline{x})(\exists \underline{w}''\in \underline{w})[(\underline{w}'\ne \underline{x}) \vee \psi_{0}(\tup x, \tup w'', \tup a)].\\
&\equiv(\forall^{\st}\underline{w}')(\exists^{\st} \underline{w})(\exists \underline{w}''\in \underline{w}) \psi_{0}(\tup w', \tup w'', \tup a).
\end{align*}
The last step is obtained by taking $\underline{x}=\underline{w}'$.  Hence, we may conclude that the normal form $(\forall^{\st} \underline{x})(\exists^{\st}y)\psi_{0}(\underline{x}, \tup y, \tup a)$ is invariant under $S_{\st}$, and we are done.    \qed
\end{proof}
Note that the previous proof may also be found in \cite{samzoo, samGH, sambon}.  
\subsection{Full proof of Theorem \ref{proto}}\label{kulll}
In this section, we establish the part of Theorem \ref{soareyouuuuuu} not covered by the proof in Section \ref{FUWKL}.  
\begin{proof}
We first prove $\paai \di \UWKL^{+}$ in $\RCAO$.  To this end, consider the functional $\nu^{2}$ defined as follows:
\[
\nu(f^{1}, M^{0}):=
\begin{cases}
(\mu n \leq M)f(n)=0 & \text{ if such exists} \\
0 & \text{ otherwise} 
\end{cases}.
\]
Thanks to $\paai$, we have $(\forall^{\st}f^{1})(\forall M, N\in \Omega)(\nu(f,M)=_{0}\nu(f, N))$, and applying underspill (which is available in $\RCAO$ due to \cite{brie}*{\S5.3}), we obtain
\be\label{labbe}
(\forall^{\st}f^{1})(\exists^{\st} n^{0})(\forall M, N\geq n)(\nu(f,M)=_{0}\nu(f, N)),
\ee
and applying $\HAC_{\INT}$ to \eqref{labbe} yields $\Psi^{1\di 0^{*}}$ and $\Phi(f):=\max_{i<|\Psi(f)|}\Psi(f)(i)$, such that $\nu(\cdot,\Phi(\cdot))$ is essentially Feferman's operator relative to `st', i.e.\ we have shown $(\mu^{2})^{\st}$.   However, $(\mu^{2})$ implies arithmetical comprehension as follows: 
\[
(\forall f^{0})\big[ (\exists n^{0})f(n)=0 \asa f(\mu(f))=0\big], 
\]
which in turn yields weak K\"onig's lemma (See \cite{simpson2}*{\S2}).  
By the previous, we have $\WKL^{\st}$, and applying $\paai$ to the consequent, we obtain that 
\be\label{spec}
(\forall^{\st}T^{1}\leq_{1}1)\big[ (\forall n)(\exists \alpha)(|\alpha|=n\wedge \alpha\in T)\di (\exists^{\st}\beta^{1}\leq_{1}1)(\forall m)(\overline{\beta}m\in T)\big]
\ee 
Now bring outside the standard quantifier `$(\exists^{\st}\beta^{1}\leq_{1} 1)$' and apply $\HAC_{\INT}$ to the resulting formula to obtain $\Psi^{1\di 1^{*}}$.
By definition, if $T$ is a standard infinite binary tree, one of the entries of $\Psi(T)$ is a (standard) path through $T$.  Thanks to $(\mu^{2})^{\st}$ and $\paai$, we can test \emph{which} entry of $\Psi$ is such a path, and define $\Phi^{1\di 1}$ to be the $\Psi(T)(i)$ with that property, where $i$ is least.  Hence, we obtain the first conjunct of $\UWKL^{+}$.  The second conjunct of the latter now easily follows by applying $\paai$ to the axiom of extensionality $\eqref{EXT}$ pertaining to $\Phi$.  
 
\medskip

By the previous, $\RCAO$ proves $\paai\di \UWKL^+$, and the latter implies, in the same was as in the proof of Theorem \ref{proto}, that
\[
(\forall^{\st} \mu^{2})\big[\textsf{\MU}(\mu)\di \eqref{spec}\big].
\]
Bringing the previous into normal form and applying Corollary \ref{consresultcor2}, one obtains a term $t$ for which one entry of $t(\mu, T)$ is a path through $T$, if $T$ is infinite and $\mu$ is as in $(\mu^{2})$.  It is trivial to use $(\mu^{2})$ to find the correct entry of $t(\mu, T)$.   %and we are done.  
\qed
\end{proof}
\subsection{Full proof of Theorem \ref{soareyouuuuuu}}\label{fappendix}
In this section, we establish the part of Theorem \ref{proto} not covered by the proof in Section \ref{group1}.  
\begin{proof}
The proof of Theorem \ref{soareyouuuuuu} hinges on the equivalence $\eqref{kral}\asa \eqref{kral2}$, 
the forward direction of which we establish now, based on \cite{simha}*{Lemma 2.3}.  %The other direction follows in exactly the same way.  

\medskip

First of all, the proof of $\WKL\di \ORD$ in $\RCA_{0}$ in the aforementioned lemma proceeds by defining for a countable torsion-free group $A$, a certain binary tree $T^{A}$ and associated set $P_{\sigma}^{A}$.  This tree $T^{A}$ is proved to be infinite, and the path provided by $\WKL$ (this being the only use of $\WKL$) is used to define the positive cone $P^{A}$ of $A$ (See Definition \ref{shiova}) via $P^{A}:=\cup_{\sigma}P^{A}_{\sigma}$, where each $\sigma$ is an initial segment of the aforementioned path.     

\medskip

Secondly, in light of the previous paragraph, the binary tree $T^{A}$ and set $P_{\sigma}^{A}$ may be defined in $\RCAO$.  By the definition of these objects in \cite{simha}*{p.\ 178}, they are standard.  Furthermore, since all recursor constants are standard (by Definition~\ref{debs}), we may use $\Pi_{1}^{0}$-induction and its relativisation to `st'.  Hence, $\RCAO$ also proves that the tree $T$ is infinite relative to `\st' (in the sense that it includes binary type zero sequences of any standard length).  

\medskip

Thirdly, working in $\RCAO+\eqref{kral}$ let $A$ be a standard countable abelian group such that $(\forall^{\st}n, a\in A)(n\times a\ne 0_{A})$, i.e.\ $A$ is torsion-free relative to `st'.  By the contraposition of \eqref{kral}, the tree $T^{A}$ is such that $(\forall^{\st}g^{2})(\exists \beta\leq_{1}1)\overline{\beta}g(\beta) \in T^{A}$.    
For standard $h^{2}$, let $\beta^{1}\leq_{1}1$ be such a sequence and define $P^{A,h}:=\cup_{\sigma \preceq \overline{\beta}3h(\beta)+3}P_{\sigma}$.  
Clearly, $P^{A, h}$ is a `partial' positive cone of $A$, and we have $ (\exists X^{1})(\forall a, b, c \in A)(a, b, c \leq h(X) \di a+c \leq_{X} b+c)$.  Hence, we have proved that for standard $A^{1},h^{2}$, we have the contraposition of \eqref{kral2}.  In short, $\RCAO\vdash \eqref{kral}\di \eqref{kral2}$.  

\medskip

Finally, the reverse implication can be proved in exactly the same way based on the proof of \cite{simha}*{Theorem 2.5}.  However, this involves proving a version of $\eqref{kral}\asa \eqref{kral2}$ for $\FAN$ and \cite{simpson2}*{IV.4.4.3}, where the latter is the statement that the ranges of two non-overlapping functions can be separated by a set.  Although this is quite straightforward, it is beyond the scope of this appendix.      
\qed
\end{proof}

\begin{biblist}
\bib{avi2}{article}{
  author={Avigad, Jeremy},
  author={Feferman, Solomon},
  title={G\"odel's functional \(``Dialectica''\) interpretation},
  conference={ title={Handbook of proof theory}, },
  book={ publisher={North-Holland}, },
  date={1998},
  pages={337--405},
}

\bib{brie}{article}{
  author={van den Berg, Benno},
  author={Briseid, Eyvind},
  author={Safarik, Pavol},
  title={A functional interpretation for nonstandard arithmetic},
  journal={Ann. Pure Appl. Logic},
  volume={163},
  date={2012},
  pages={1962--1994},
}

\bib{simha}{article}{
  author={Hatzikiriakou, Kostas},
  author={Simpson, Stephen G.},
  title={${\rm WKL}_0$ and orderings of countable abelian groups},
  conference={ date={1987}, },
  book={ series={Contemp. Math.}, volume={106}, },
  date={1990},
  pages={177--180},
}

\bib{kohlenbach2}{article}{
  author={Kohlenbach, Ulrich},
  title={Higher order reverse mathematics},
  conference={ title={Reverse mathematics 2001}, },
  book={ series={Lect. Notes Log.}, volume={21}, publisher={ASL}, },
  date={2005},
  pages={281--295},
}

\bib{kooltje}{article}{
  author={Kohlenbach, Ulrich},
  title={On uniform weak K\"onig's lemma},
  journal={Ann. Pure Appl. Logic},
  volume={114},
  date={2002},
  number={1-3},
  pages={103--116},
}

\bib{wownelly}{article}{
  author={Nelson, Edward},
  title={Internal set theory: a new approach to nonstandard analysis},
  journal={Bull. Amer. Math. Soc.},
  volume={83},
  date={1977},
  number={6},
  pages={1165--1198},
}

\bib{samGH}{article}{
  author={Sanders, Sam},
  title={The Gandy-Hyland functional and a hitherto unknown computational aspect of Nonstandard Analysis},
  year={2015},
  journal={Submitted, \url {http://arxiv.org/abs/1502.03622}},
}

\bib{samzoo}{article}{
  author={Sanders, Sam},
  title={The taming of the Reverse Mathematics zoo},
  year={2015},
  journal={Submitted, \url {http://arxiv.org/abs/1412.2022}},
}

\bib{sambon}{article}{
  author={Sanders, Sam},
  title={The unreasonable effectiveness of Nonstandard Analysis},
  year={2015},
  journal={Submitted, \url {http://arxiv.org/abs/1508.07434}},
}

\bib{fegas}{article}{
  author={Sanders, Sam},
  title={Non-standard Nonstandard Analysis and a hitherto unknown computational aspect of Nonstandard Analysis},
  year={2015},
  journal={Submitted, Available from arXiv: \url {http://arxiv.org/abs/1509.00282}},
}

\bib{simpson1}{collection}{
  title={Reverse mathematics 2001},
  series={Lecture Notes in Logic},
  volume={21},
  editor={Simpson, Stephen G.},
  publisher={ASL},
  place={La Jolla, CA},
  date={2005},
  pages={x+401},
}

\bib{simpson2}{book}{
  author={Simpson, Stephen G.},
  title={Subsystems of second order arithmetic},
  series={Perspectives in Logic},
  publisher={CUP},
  date={2009},
  pages={xvi+444},
}


%\bibselect{allkeida}
\end{biblist}
\bye